\documentclass{article}

\usepackage[latin1]{inputenc}  

\usepackage{algorithm}
\usepackage{algorithmic}

\usepackage{amsmath,amssymb,dcolumn,amsthm}
\usepackage{graphicx,color}

\usepackage{tikz}
\usetikzlibrary{arrows}
\tikzstyle{block} = [draw, draw=black, line width = 1pt, rectangle,
    rounded corners=4pt,
    minimum height=3em, text width=.7\columnwidth,
    inner sep = 10pt]

\usepackage{pgfplots} 
\pgfplotsset{compat=newest} 
\pgfplotsset{plot coordinates/math parser=false} 
\newlength\figureheight 
\newlength\figurewidth
  
\usepackage{caption}
\usepackage{subcaption}

\usepackage{ian_misc,ian_opt,ian_math,ian_IEEEarticle,ian_abbrvs,ian_ricupd}

\begin{document}


\title{ Low-Rank Modifications of Riccati Factorizations for Model Predictive Control}

\author{Isak Nielsen and Daniel Axehill%
  \thanks{I. Nielsen and D. Axehill are with the Division of Automatic
    Control, Linköping University, SE-58183 Linköping, Sweden, \texttt{isak.nielsen@liu.se, daniel@isy.liu.se}.}}



\maketitle

\begin{abstract}
In Model Predictive Control (\MPC) the control input is computed by solving a constrained finite-time optimal control (\cftoc) problem at each sample in the control loop. The main computational effort is often spent on computing the search directions, which in \MPC corresponds to solving unconstrained finite-time optimal control (\uftoc) problems. This is commonly performed using Riccati recursions or generic sparsity exploiting algorithms. In this work the focus is efficient search direction computations for active-set (\AS) type methods. The system of equations to be solved at each \AS iteration is changed only by a low-rank modification of the previous one, and exploiting this structured change is important for the performance of \AS type solvers. In this paper, theory for how to exploit these low-rank changes by modifying the Riccati factorization between \AS iterations in a structured way is presented. A numerical evaluation of the proposed algorithm shows that the computation time can be significantly reduced by modifying, instead of re-computing, the Riccati factorization. This speed-up can be important for \AS type solvers used for linear, nonlinear and hybrid \MPC.

\end{abstract}

%


\section{Introduction}

Model Predictive Control (\MPC) is a control strategy where the applied control input is computed by minimizing a cost function while satisfying constraints on the states and control inputs. It has become one of the most widely used advanced control strategies in industry, and some important reasons for its popularity are that it can handle  multivariable systems and constraints on states and control inputs in a structured way,~\cite{maciejowski2002predictive}. In each sample of the \MPC control loop a constrained finite-time optimal control (\cftoc) problem is solved on-line, which creates a need for efficient optimization algorithms. Note that similar linear algebra is also useful in off-line applications such as explicit \MPC solvers. The \MPC problem and the corresponding \cftoc problem can be of different types depending on which system and problem formulation that is used. Some common variants are linear \MPC, nonlinear \MPC and hybrid \MPC. In many cases the main computational effort when solving the \cftoc{} problem boils down to compute the search directions, which corresponds to solving unconstrained finite-time optimal control (\uftoc) problems. The \uftoc{} problems can be solved using for example Riccati recursions, and some examples of how optimization routines have been sped up by using Riccati recursions are~\cite{jonson83:thesis,rao98:_applic_inter_point_method_model_predic_contr,hansson00:_primal_dual_inter_point_method,
vandenberghe02:_robus_full,axehill06:_mixed_integ_dual_quadr_progr_tail_mpc,axevanhan07:_relax_applic_mipc_compa_and_eff_compu,axehill08:thesis,
axehill08:_dual_gradien_projec_quadr_progr,diehl09:_nonlin_model_predic_contr,nielsen13low_rank_updates,nielsen15:parallel_factorization_cdc,frison16:thesis}.

The use of Riccati recursions in active-set (\AS) methods for optimal control can be found as early as in~\cite{jonson83:thesis}. In this reference a Riccati recursion is used to factor the major block of the \KKT matrix, and for the other block standard low-rank modifications of factorizations are used on a dense system of equations of the size of the number of active inequality constraints. The computational complexity of this algorithm grows quadratically in the number of active inequality constraints.  An alternative sparse non-Riccati factorization is used in~\cite{kirches2011factorization}, and the factorization is updated after changes in the \AS iterations.
 
 In \AS methods it is often crucial to modify the factorization of the \KKT matrix instead of re-factorizing it between \AS iterations,~\cite{nocedal06:num_opt}. Since this has traditionally not been considered possible when using the Riccati factorization, it has sometimes been argued that this factorization is not suitable for \AS methods. However, in~\cite{nielsen13low_rank_updates} a method for making low-rank modifications of the Riccati factorization by exploiting the structured changes between \AS iterations was introduced, showing that this is indeed possible. The work in~\cite{nielsen13low_rank_updates} is limited to problems with non-singular control input weight matrices and simple control input bounds, and modifications of the \KKT matrix is only possible at a single time index.
 
%


The main contribution in this paper is the extension of the theory in~\cite{nielsen13low_rank_updates} to handle more general forms of \uftoc problems, where the \KKT matrix can be singular. The derivation of this result looks similar to the one in~\cite{nielsen13low_rank_updates}, but here more technical depth is added since additional mathematical tools are needed in this paper to cope with the singularity of the \KKT matrix. In this paper it is also described how to modify the factorization for more general modifications of the \KKT matrix where constraints are simultaneously added (or removed) at different time indices. Both these generalizations can  be important when using for example dual projection \AS solvers like the one in~\cite{axehill08:thesis}. Furthermore, in~\cite{nielsen13low_rank_updates} only bound constraints on the control inputs are considered in the \cftoc problem, whereas it will be shown in this paper how the theory can be applied to problems with both state and control input constraints. A more detailed description of the material presented in this paper can be found in the thesis in~\cite{nielsen15:licthesis}.



In this article, $\posdefmats^n$ ($\possemidefmats^n$) denotes
symmetric positive (semi) definite matrices with $n$ columns, $\intset{i}{j} = \braces{z \in \mathbb{Z} \; | \; i \leq z \leq j}$, and $\range{A}$ denotes the range space of a matrix $A$.

Section~\ref{sec:upd:prob} introduces the \cftoc problem and Section~\ref{sec:upd:opt} optimization preliminaries. In Section~\ref{sec:upd:lr_modifications} the main result is presented, and in Section~\ref{sec:upd:general_constr} it is shown how to use this for more general problems. The numerical results and conclusion are presented in Section~\ref{sec:upd:numerical_results} and Section~\ref{sec:upd:conclusions}, respectively.


\section{Problem Formulation}
\label{sec:upd:prob}
For linear \MPC problems, the corresponding \cftoc problem consists of a quadratic objective function and affine dynamics constraints. For now, consider only upper and lower bounds on the control signal. Let $t$ denote the time-index in the \MPC optimization problem (\ie{}, $t=0$ is the current time), $N$ the prediction horizon, $x_t \in \realnums{\nx} $ the state, $u_t \in \realnums{\nuut{t}}$ the control input, $\xinit \in \realnums{\nx}$ the initial state, and
\begin{equation}
\timestack{x} \triangleq \bracks{x_0^T,\ldots,x_N^T}^T, \quad \timestack{u} \triangleq \bracks{u_0^T,\ldots, u_{N-1}^T}^T,
\end{equation}
the stacked states and control inputs, respectively. The \cftoc problem can then be written in the form
%
\begin{equation}
\minimizes{&\sum_{t=0}^{N-1}\Big( \frac{1}{2}\begin{bmatrix}
x_t \\ u_t
\end{bmatrix}^T\Qpart{t} \begin{bmatrix}
x_t \\ u_t
\end{bmatrix} +  \lin{t}^T \begin{bmatrix}
x_t\\u_t
\end{bmatrix} +\cpart{t} \Big) \\&+\frac{1}{2}x_N^T\Qx{N}x_N + \lin{x,N}^Tx_N+\cpart{N} }{\timestack{x},\timestack{u}}{&x_0 = \xinit \\ &x_{t+1} = A_t x_t + B_tu_t + a_t, \; t \in \intset{0}{N-1} \\ &\umin{t} \preceq u_t \preceq \umax{t}, \; t \in \intset{0}{N-1},} \label{eq:upd:primal_problem}
\end{equation}
where $\Qx{N} \in \possemidefmats^{\nx}$ and $\Qpart{t} \in \possemidefmats^{\nx + \nuut{t}}$. Let the matrix $\Qpart{t}$ and the vector $\lin{t} \in \realnums{\nx +\nuut{t}}$ be partitioned as
\begin{equation}
\Qpart{t} = \begin{bmatrix}
\Qx{t} & \Qxu{t} \\ \Qxu{t}^T & \Qu{t}
\end{bmatrix}, \quad \lin{t} = \begin{bmatrix}
\bar{l}_{x,t} \\ \bar{l}_{u,t}
\end{bmatrix}, \; t \in \intset{0}{N-1}.
\end{equation} 
Note that the more common additional assumption $\Qu{t} \in \posdefmats^{\nuut{t}}$ is not used in this problem formulation in order to, for example, be able to represent dual \MPC problems.

Furthermore, define $\lambda_{t+1}$ as the dual variable for the equality constraint $-x_{t+1} + A_t x_t + B_t u_t + a_t = 0$ and $\mu_t$ as the dual variable for the inequality constraint
\begin{equation}
\begin{bmatrix}
I\\-I
\end{bmatrix} u_t - \begin{bmatrix}
\umax{t} \\ - \umin{t}
\end{bmatrix} \preceq 0.
\end{equation} 
%
%

%


\section{Optimization Preliminaries}
\label{sec:upd:opt}
%

The \cftoc problem~\eqref{eq:upd:primal_problem} is a convex quadratic programming (\QP) problem. Hence, it can be solved using several different types of optimization methods, where one common type is \AS methods, see for example,~\cite{boyd04:_convex_optim,nocedal06:num_opt}.

\subsection{Active-set QP methods}
\AS methods solve a \QP problem by determining the set of constraints that are active, \ie, hold with equality, at the optimal solution. This set of active constraints is denoted the optimal \emph{active set}, and an \AS solver operates by finding this set of constraints iteratively,~\cite{nocedal06:num_opt}. Since the optimal active set is usually not known a priori, an \AS solver starts with an initial set of constraints that are forced to hold with equality, and then changes this so-called \emph{working set} by adding or removing constraints until the optimal active set has been determined. These modifications of the working set are usually relatively small and the modifications of the corresponding \KKT matrix between \AS iterations are thus of low rank. The modification techniques presented in this work can be used both by traditional \AS solvers where one constraint is added or removed to the working set at each iteration, and for solvers that add or remove several constraints to the working set at each iteration such as those presented in~\cite{nocedal06:num_opt,axehill08:thesis,axehill08:_dual_gradien_projec_quadr_progr}.

Let $\wset{j}$ denote the subset of the working set that contains the indices of the inequality constraints that temporarily hold with equality at \AS iteration $j$, and let $\wsetc{j}$ denote the set of inequality constraints that are temporarily disregarded at \AS iteration $j$. In problem~\eqref{eq:upd:primal_problem} only control input constraints are used, and hence forcing a constraint to hold with equality corresponds to removing that control input as an optimization variable from the optimization problem and treating it as a constant. Similarly, by disregarding an inequality constraint, the corresponding control input becomes unconstrained. This can be formalized by introducing $\ufree{t}$ as the free part of the control inputs and $\ufixed{t}$ as the fixed part as follows 
\begin{equation}
\ufree{t} \triangleq u_{t}(\wsetc{j}), \; \ufixed{t} \triangleq u_{t}(\wset{j}), \; u_t = \Pi_t \begin{bmatrix} \label{eq:upd:u_part}
\ufree{t} \\ \ufixed{t}
\end{bmatrix}\!\!, \; t \!\in\! \intset{0}{N-1}, \! \!
\end{equation}
where $\Pi_t$ is a permutation matrix satisfying $\Pi_t^T \Pi_t = I$. Here $u_t(\wset{j})$ is used to denote the control inputs at time $t$ with corresponding constraints in $\wset{j}$. Using this notation, $B_t$, $\Qu{t}$, $\Qxu{t}$ and $\bar{l}_{u,t}$  can be partitioned consistently with $\ufree{t}$ and $\ufixed{t}$: 
\begin{subequations}
\begin{align}
 &\begin{bmatrix}
\Bw{t} & \Bv{t}
\end{bmatrix} \Pi_t^T  \triangleq B_{t}, \; \begin{bmatrix}
\Qxw{t} & \Qxv{t}
\end{bmatrix} \Pi_t^T \triangleq \Qxu{t}, \label{eq:upd:B_part}\\
&\Pi_t \begin{bmatrix}
\Qw{t} & \Qwv{t} \\ \Qwv{t}^T & \Qv{t}
\end{bmatrix} \Pi_t^T \triangleq \Qu{t}, \; \Pi_t\begin{bmatrix}
\bar l_{w,t} \\ \bar l_{v,t}
\end{bmatrix}  \triangleq \bar{l}_{u,t}. \label{eq:upd:Qu_part}
\end{align}
\end{subequations}
By using this partitioning of the control input and the corresponding matrices, the \uftoc{} problem that is solved at \AS iteration $j$ to compute the search direction is given by
\begin{equation}
\minimizest{&\!\!\sum_{t=0}^{N-1}\!\Big( \frac{1}{2}\begin{bmatrix}
x_t \\ \ufree{t}
\end{bmatrix}^T\! \!\begin{bmatrix}
\Qx{t}\! \! &\! \! \Qxw{t} \\ \Qxw{t}^T \! \!&\! \! \Qw{t}
\end{bmatrix} \! \! \begin{bmatrix}
x_t \\ \ufree{t}
\end{bmatrix}\! + \! \begin{bmatrix}
\linx{t} \\ \linw{t}
\end{bmatrix}^T \! \! \begin{bmatrix}
x_t\\\ufree{t}
\end{bmatrix} \!+\!\cv{t} \Big) \\&\!\!+\frac{1}{2}x_N^T\Qx{N}x_N + \lin{x,N}^Tx_N+\cpart{N} }{\timestack{x},\timestack{w}}{&x_0 = \xinit \\ &x_{t+1} = A_t x_t + \Bw{t}\ufree{t} + \av{t}, \; t \in \intset{0}{N-1},} \label{eq:upd:uftoc_prob}
\end{equation}
where
\begin{subequations}
\begin{align}
\linx{t} &\triangleq \bar l_{x,t} + \Qxv{t}\ufixed{t}, \quad \linw{t} \triangleq \bar l_{w,t} + \Qwv{t}\ufixed{t}, \\
\cv{t} &\triangleq c_t + \bar l_{v,t}^T\ufixed{t} + \frac{1}{2} \ufixed{t}^T \Qv{t}\ufixed{t}, \quad \av{t}  \triangleq a_t + \Bv{t}\ufixed{t}.
\end{align}
\end{subequations}
Computing the sequence of search directions in an \AS type solver applied to the \cftoc problem~\eqref{eq:upd:primal_problem} hence corresponds to solving a sequence of \uftoc problems in the form~\eqref{eq:upd:uftoc_prob}.


\subsection{Standard Riccati recursion}
\label{subsec:standard_riccati}
The solution to the \uftoc{} problem~\eqref{eq:upd:uftoc_prob} is computed by solving a set of linear equations known as the \KKT optimality conditions. The special structure of the \uftoc{} problem considered in this work corresponds to a sparse, almost block diagonal, \KKT system which can be solved very efficiently using a Riccati recursion, see, \eg,~\cite{jonson83:thesis,rao98:_applic_inter_point_method_model_predic_contr,axehill08:thesis,nielsen15:licthesis}. The Riccati recursion consists of a factorization of the \KKT matrix (Algorithm~\ref{alg:upd:ric_fac}), followed by back- and forward substitutions  (algorithms~{\ref{alg:upd:ric_rec_bw_rec}-\ref{alg:upd:dual_var_fw_rec}}) for computing the solution to~\eqref{eq:upd:uftoc_prob},~\cite{vandenberghe02:_robus_full}. Algorithm~\ref{alg:upd:ric_fac}, which is the computationally demanding part of the Riccati recursion, computes the variables $P_{t+1} \in \possemidefmats^{\nx}$ and $K_{t+1} \in \realnums{\nwt{t}\times \nx}$ using the auxiliary variables
\begin{equation}
M_{t+1} \! \triangleq\!  \begin{bmatrix}
F_{t+1}\!\! &\!\! H_{t+1} \\ H_{t+1}^T\!\! &\!\! G_{t+1}
\end{bmatrix}\!\! \triangleq\!\! \begin{bmatrix}
\Qx{t}\! \! &\! \! \Qxw{t} \\ \Qxw{t}^T \! \!&\! \! \Qw{t}
\end{bmatrix} + \begin{bmatrix}
A_t^T \\ \Bw{t}^T
\end{bmatrix}\! P_{t+1}\! \begin{bmatrix}
A_t^T \\ \Bw{t}^T
\end{bmatrix}^T \!\!\!\!\!, \label{eq:upd:M}
\end{equation}
where $M_{t+1} \in \possemidefmats^{\nx \times \nwt{t}}$, $F_{t+1} \in \possemidefmats^{\nx}$, $G_{t+1} \in \possemidefmats^{\nwt{t}}$ and $H_{t+1} \in \realnums{\nx \times \nwt{t}}$  by construction. Since $\Qw{t} \in \possemidefmats^{\nwt{t}}$ it follows that also $G_{t+1}\in \possemidefmats^{\nwt{t}}$. When one (or more) $G_{t+1}$ is singular a non-unique Riccati factorization still exists, but the solution of the \KKT system is either non-unique or non-existing,~\cite{axehill08:thesis,nielsen15:licthesis}. How to handle this case is determined at the solver level, and one way to do this is presented in~\cite{axehill08:thesis}.

Using the Riccati recursion to compute the solution to~\eqref{eq:upd:uftoc_prob} requires $\Ordo{N}$ complexity, compared to $\Ordo{N^3}$ or $\Ordo{N^2}$ for dense solvers that re-factorize, or modify the factorizations of, the \KKT matrix without exploiting the \uftoc problem structure, respectively. For more information, see, \eg,~\cite{nocedal06:num_opt}.

\begin{algorithm}
  \caption{Riccati Factorization} \label{alg:upd:ric_fac}
  \begin{algorithmic}[1]
    \STATE $P_N := \Qx{N}$
    \FOR{$t=N-1,\ldots,0$}
    \STATE $F_{t+1} := \Qx{t} + A_t^TP_{t+1}A_{t}$\label{alg:upd:ric_fac:line:F_comp} \\
    \STATE $G_{t+1} := \Qw{t} + \Bw{t}^TP_{t+1}\Bw{t}$ \label{alg:upd:ric_fac:line:G_comp}\\
    \STATE $H_{t+1} := \Qxw{t} + A_{t}^TP_{t+1}\Bw{t}$\label{alg:upd:ric_fac:line:H_comp} \\
    \STATE Compute and store a factorization of $G_{t+1}$.\label{alg:upd:ric_fac:line:factorize_G}
    \STATE Compute a solution $K_{t+1}$ to: \\  
    $G_{t+1}K_{t+1} = -H_{t+1}^T$ \label{alg:upd:ric_fac:line:GKeqHT}\\
    \STATE $P_{t} := F_{t+1} - K_{t+1}^TG_{t+1}K_{t+1}$\label{alg:upd:ric_fac:line:P_comp}
    \ENDFOR
  \end{algorithmic}
\end{algorithm}
\vspace{-2mm}

\begin{algorithm}[ht]
  \caption{Backward recursion}
  \label{alg:upd:ric_rec_bw_rec}
  \begin{algorithmic}[1]
    \STATE $\Psi_N := -\linx{N}$, \quad $\cb{N} := c_N$
    \FOR{$t = N-1,\hdots,0$}
    \STATE Compute a solution $\uaff{t+1}$ to: \\ 
     $G_{t+1}k_{t+1} = \Bw{t}^T\Psi_{t+1} -
        \linw{t} - \Bw{t}P_{t+1} \av{t}$
    \STATE $\Psi_{t} := A_t^T\Psi_{t+1}
    - H_{t+1}\uaff{t+1} - \linx{t} -A_t^TP_{t+1}\av{t}$ \\
    \STATE $\cb{t}:= \cb{t+1} +\cv{t} + \frac{1}{2}\av{t}^TP_{t+1}\av{t} - \Psi_{t+1}^T\av{t}- \frac{1}{2}\uaff{t+1}^T G_{t+1}\uaff{t+1}$
    \ENDFOR
  \end{algorithmic}
\end{algorithm}
\vspace{-2mm}

\begin{algorithm}[ht]
  \caption{Forward recursion}
  \label{alg:upd:ric_rec_fw_rec}
  \begin{algorithmic}[1]
  	\STATE $x_0 := \xinit$
    \FOR{$t = 0,\hdots,N-1$}
    \STATE $w_t := \uaff{t+1} + K_{t+1}x_t$
    \STATE $x_{t+1} := A_tx_t + \Bw{t} w_t + \av{t}$
    \STATE $\lambda_t := P_{t}x_t - \Psi_t$
    \ENDFOR
    \STATE $\lambda_N := P_Nx_N -\Psi_N$
  \end{algorithmic}
\end{algorithm}

\vspace{-2mm}

\begin{algorithm}[ht]
  \caption{Forward recursion (Dual variables)}
  \label{alg:upd:dual_var_fw_rec}
  \begin{algorithmic}[1]
    \FOR{$t = 0,\hdots,N-1$}
    \STATE $\mu_t({\wsetc{j}}) := 0$ 
    \STATE $\mu_t(\wset{j}) := \bar{l}_{v,t}+ \Qxv{t}^Tx_t  + \Qwv{t}^Tw_{t}+ \Qv{t}\ufixed{t} + \Bv{t}^T\lambda_{t+1}
    $ 
    \ENDFOR
  \end{algorithmic}
\end{algorithm}


\section{Low-Rank Modification of the \\Riccati Factorization}
\label{sec:upd:lr_modifications}
A standard approach to improve the performance of an \AS solver is to modify the factorization of the \KKT matrix instead of re-factorizing it between \AS iterations,~\cite{nocedal06:num_opt}. Here it will be shown how to modify the Riccati factorization (Algorithm~1) between \AS iterations when solving a \cftoc problem~\eqref{eq:upd:primal_problem}. Since $\Qu{t} \in \possemidefmats^{\nuut{t}}$, and hence possibly also $\Qw{t} \in \possemidefmats^{\nwt{t}}$, the \KKT matrix for the \uftoc problem that is solved to compute the search direction can be singular (some $G_{t+1}$ in Algorithm~\ref{alg:upd:ric_fac} can be singular). The derivations in this section are similar to the one in~\cite{nielsen13low_rank_updates}, but the extension presented here adds more technical depth since additional mathematical tools such as generalized Schur complements (\GSCs), the quotient formula for \GSCs and the Moore-Penrose pseudo-inverse are required to cope with the possibly singular \KKT matrix. For a detailed description of these, see for example~\cite{albert69,carlsonPseudo,zhang2005schur}. 

Furthermore, in~\cite{nielsen13low_rank_updates} it was only shown how to modify the Riccati factorization when modifying the working set at a single time instance. If constraints at different time indices are added or removed  in the same \AS iteration, the factorization can be modified by performing a sequence of complete modifications. However, in this section it will be shown how to handle \emph{either} adding \emph{or} removing several constraints at \emph{different} time indices by instead gradually increasing the size of the modification of the factorization.   If the solver both adds and removes constraints in the same \AS iteration, the factorization must be modified sequentially. Note that as the size of the modification increases, it might be better to re-compute the remaining part of the factorization from scratch.
Which approach that is most efficient depends on for example the size of the modification, and can be investigated off-line. That work is however outside the scope of this paper.

By introducing the \GSC operator as $\pschur{}{}$, the \GSC with respect to $G_{t+1}$ of $M_{t+1}$ in~\eqref{eq:upd:M}
is $\pschur{M_{t+1}}{G_{t+1}} \triangleq F_{t+1} - H_{t+1}\pinv{G_{t+1}}H_{t+1}^T$.
 Here $\pinv{G_{t+1}}$ is the Moore-Penrose pseudo inverse of $G_{t+1}$. Hence, by using Line~\ref{alg:upd:ric_fac:line:GKeqHT} in Algorithm~\ref{alg:upd:ric_fac} and basic properties of the pseudo inverse, the matrix $P_t$ in Line~\ref{alg:upd:ric_fac:line:P_comp} in Algorithm~\ref{alg:upd:ric_fac} can be calculated as
\begin{equation}
\begin{split}
P_{t} =  F_{t+1} - K_{t+1}^T G_{t+1}K_{t+1} = \pschur{M_{t+1}}{G_{t+1}}. \label{eq:upd:PeqPseudoSchur}
\end{split}
\end{equation}
%
\begin{lemma}[Quotient formula for \GSC]
\label{lem:pseude_schur_quotient}
Let the positive semi-definite matrices $\mathcal{M} \succeq 0$ and $\bar{ \mathcal{M}} \succeq 0$ be partitioned as
\begin{equation}
\mathcal{M} = \left[\begin{array}{c|cc}
A & B & C \\
& & \\[\dimexpr-\normalbaselineskip-2pt] \hline
& & \\[\dimexpr-\normalbaselineskip+2pt]
 B^T & D & E \\ C^T & E^T & F
\end{array}\right], \quad \bar{ \mathcal{M}} = \begin{bmatrix}
D & E \\ E^T & F
\end{bmatrix}.
\end{equation}
Then 
\begin{equation}
\begin{split}
&\pschur{\parens{\pschur{\mathcal{M}}{F}}}{\parens{\pschur{\bar{\mathcal{M}}}{F}}} =\pschur{\mathcal{M}}{\bar{\mathcal{ M}}} = A - B\pinv{D}B^T -\\
&\parens{C\!-\!B\pinv{D}E}\pinv{\parens{F \!-\! E^T\pinv{D}E}}\parens{C^T \! - \! E^T \pinv{D}B^T} \succeq 0, \label{eq:upd:lem:quotient_pschur}
\end{split}
\end{equation}
\begin{equation}
C^T \! -E^T \!\pinv{D}B^T \!  \in \! \range{F \! - E^T\! \pinv{D}E}\!, \; \; \; F  - E^T\pinv{D}E \succeq 0.
\end{equation}
\end{lemma}
\begin{proof}
Lemma~\ref{lem:pseude_schur_quotient} follows directly from Theorem 4 in~\cite{carlsonPseudo}. The details are given in the proof of Lemma 4.1 in~\cite{nielsen15:licthesis}.
\end{proof}
%

 In this paper, a tilde will be used to indicate a matrix that has been modified. Hence, the modified version of some matrix $X$ is denoted $\tilde X$.



\subsection{Sequence of low-rank modifications}
\label{subsec:upd:impact_on_subseq_iter}
Assume that $\Pt{t+1}$ for some $t \in \intset{0}{N-1}$ is a \emph{downdate} of $P_{t+1}$, given by (the superscript ``$-$'' indicates a downdate)
%
\begin{equation}
\Pt{t+1} = P_{t+1} - \Vm{t+1}{\Cmpinv{t+1}} \Vm{t+1}^T \in \possemidefmats^{\nx}, \label{eq:upd:Pt_compact_dwn_tm_before_add_rem}
\end{equation}
with $\Cm{t+1} \in \possemidefmats^{\kprev}$, $\Vm{t+1} \in \realnums{\nx \times \kprev}$ and $\Vm{t+1}^T \in \range{\Cm{t+1}}$. Later in this section, and in sections~\ref{subsec:upd:remove_control_input_constr} and~\ref{subsec:upd:add_control_input_constr}, Lemma~\ref{lem:pseude_schur_quotient} will be used to show that this assumption holds for all modifications presented in this paper. The downdate is considered to be of low rank if $\kprev < \nx$. It will now be shown how this {downdate} of $P_{t+1}$ affects the matrices in the Riccati factorization for the time-steps $\tau \in \intset{0}{t}$. By substituting $P_{t+1}$ in lines~\ref{alg:upd:ric_fac:line:F_comp}-\ref{alg:upd:ric_fac:line:H_comp} in Algorithm~\ref{alg:upd:ric_fac} with $\Pt{t+1}$ from~\eqref{eq:upd:Pt_compact_dwn_tm_before_add_rem}, straightforward calculations give
\vspace{-3pt}
\begin{subequations}
\label{eq:upd:def_Ft_Gt_Ht_subseq_dwn}
\begin{align}
&\Ft{t+1}  = F_{t+1} - A_t^T\Vm{t+1}\Cmpinv{t+1}\Vm{t+1}^T A_t, \\
&\Gt{t+1} = G_{t+1} - \Bw{t}^T \Vm{t+1} {\Cmpinv{t+1}}\Vm{t+1}^T\Bw{t}, \label{eq:upd:def_Gt_subseq_dwn} \\
&\Ht{t+1} = H_{t+1} - A_t^T \Vm{t+1} {\Cmpinv{t+1}}\Vm{t+1}^T \Bw{t}. \label{eq:upd:def_Ht_subseq_dwn}
\end{align}
\end{subequations}
The equations in~\eqref{eq:upd:def_Ft_Gt_Ht_subseq_dwn} can be written in matrix form as
\begin{equation}
\begin{split}
\Mt{t+1} &= \begin{bmatrix}
\Ft{t+1} & \Ht{t+1} \\ \Ht{t+1}^T & \Gt{t+1}
\end{bmatrix} = \begin{bmatrix}
F_{t+1} & H_{t+1} \\ H_{t+1}^T & G_{t+1}
\end{bmatrix} - \\&\begin{bmatrix}
A_t^T \Vm{t+1} \\ \Bw{t}^T \Vm{t+1}
\end{bmatrix}  {\Cmpinv{t+1}} \begin{bmatrix}
A_t^T \Vm{t+1} \\ \Bw{t}^T \Vm{t+1}
\end{bmatrix}^T \in \possemidefmats^{\nx + \nwt{t}}. \label{eq:upd:Mt_written_out_sub_iter}
\end{split}
\end{equation}
Since $\Pt{t+1} \in \possemidefmats^{\nx}$ and $\Mt{t+}$ is defined as in~\eqref{eq:upd:M}, it follows that $\Mt{t+}$ is positive semidefinite by construction. Now, define 
\begin{align}
\Mh{t+1} &\triangleq \left[\begin{array}{c|cc}
F_{t+1} & H_{t+1} & A_t^T\Vm{t+1} \\
& & \\[\dimexpr-\normalbaselineskip+2pt] \hline
& & \\[\dimexpr-\normalbaselineskip+3pt]
H_{t+1}^T & G_{t+1} & \Bw{t}^T\Vm{t+1} \\
\Vm{t+1}^TA_t & \Vm{t+1}^T \Bw{t} & \Cm{t+1}
\end{array}\right],
\end{align}
and let $\Mb{t+1}$ be the second diagonal block of $\Mh{t+1}$. Note that $\Gt{t+1} = \pschur{\Mb{t+1}}{\Cm{t+1}}$.
From~\cite{albert69} it follows that $\Mh{t+1} \succeq 0$ since $\Cm{t+1} \succeq 0$ and $\Vm{t+1}^T \in \range{\Cm{t+1}}$ by assumption, and $\Mt{t+1} = \pschur{\Mh{t+1}}{\Cm{t+1}} \succeq 0$ in~\eqref{eq:upd:Mt_written_out_sub_iter}. Hence also $\Mb{t+1} \succeq 0$, and Lemma~\ref{lem:pseude_schur_quotient} thus can be used to compute $\Pt{t} = \pschur{\Mt{t+1}}{\Gt{t+1}} = \pschur{\Mh{t+1}}{\Mb{t+1}}$ (first equality in~\eqref{eq:upd:lem:quotient_pschur}). By using the block partitioning $A=F_{t+1}$, $B=H_{t+1}$, $C=A_t^T \Vm{t+1}$, $D=G_{t+1}$, $E=\Bw{t}^T \Vm{t+1}$ and $F = \Cm{t+1}$, the modified version of $P_{t}$ is computed using the second equality in~\eqref{eq:upd:lem:quotient_pschur} in Lemma~\ref{lem:pseude_schur_quotient} as
\begin{equation}
\Pt{t} = P_t - \Vm{t} {\Cmpinv{t}} \Vm{t}^T \in \possemidefmats^{\nx}, \label{eq:upd:Pt_written_out_sub_iter_dwndate}
\end{equation}
where
\begin{subequations}
\begin{align}
\Vm{t} &\triangleq \parens{A_t^T - H_{t+1}\pinv{G_{t+1}}\Bw{t}^T}\Vm{t+1} \in \realnums{\nx \times \kprev}, \\
\Cm{t} &\triangleq \Cm{t+1} - \Vm{t+1}^T \Bw{t} \pinv{G_{t+1}} \Bw{t}^T \Vm{t+1} \in \possemidefmats^{\kprev}, \\
\Vm{t}^T &\in \range{\Cm{t}} .
\end{align}
\end{subequations}
%
Using similar calculations, an \emph{update} of $P_{t+1}$ in the form
\begin{equation}
\begin{split}
\Pt{t+1} &= P_{t+1} + \Vp{t+1}{\Cppinv{t+1}} \Vp{t+1}^T \in \possemidefmats^{\nx} \iff \\ P_{t+1} &= \Pt{t+1} - \Vp{t+1}{\Cppinv{t+1}} \Vp{t+1}^T \in \possemidefmats^{\nx}, \label{eq:upd:Pt_compact_upd_tm_before_add_rem}
\end{split}
\end{equation}
with $\Cp{t+1} \in \possemidefmats^{\kprev}$, $\Vp{t+1} \in \realnums{\nx \times \kprev}$ and $\Vp{t+1}^T \in \range{\Cp{t+1}}$, can be shown to result in the update
\begin{equation}
\Pt{t} = P_t + \Vp{t} {\Cppinv{t}}\Vp{t}^T \in \possemidefmats^{\nx}, \label{eq:upd:Pt_written_out_sub_iter_update}
\end{equation}
with (here $\Gt{t+1}$ and $\Ht{t+1}$ are defined similarly as in~\eqref{eq:upd:def_Ft_Gt_Ht_subseq_dwn})
\begin{subequations}
\label{eq:upd:def_Vp_Cp}
\begin{align}
\Vp{t} &\triangleq \parens{A_t^T -\Ht{t+1} \pinv{\Gt{t+1}} \Bw{t}^T}\Vp{t+1} \in \realnums{\nx \times \kprev}, \\ 
\Cp{t} &\triangleq \Cp{t+1} - \Vp{t+1}^T\Bw{t}\pinv{\Gt{t+1}} \Bw{t}^T \Vp{t+1} \in \possemidefmats^{\kprev}, \\
\Vp{t}^T &\in \range{\Cp{t+1}}. 
\end{align}
\end{subequations}
Note that the modified matrices $\Ht{t+1}$ and $\Gt{t+1}$ are used in~\eqref{eq:upd:def_Vp_Cp}.
Hence, a modification of $P_{t+1}$ of at most rank $\kprev$ results in a similar modification of $P_t$ of (also) at most rank $\kprev$. 
\begin{theorem}
\label{thm:upd:rank_k_modifications}
Consider a modification of at most rank $\kprev$ of $P_{\tm} \in \possemidefmats^{\nx}$ in Algorithm~\ref{alg:upd:ric_fac} at a single time instant $\tm \in \intset{1}{N}$ in either of the forms
\begin{equation}
\begin{cases}
\Pt{\tm} = P_{\tm} - \Vm{\tm} {\Cmpinv{\tm}} \Vm{\tm}^T \in \possemidefmats^{\nx} \textrm{ (downdate)} \\
\Pt{\tm} = P_{\tm} + \Vp{\tm} {\Cppinv{\tm}} \Vp{\tm}^T \in \possemidefmats^{\nx} \textrm{ (update)}
\end{cases} \label{eq:upd:Ptm_thm}
\end{equation}
where $\Cm{\tm}, \Cp{\tm} \!\!\in \possemidefmats^{\kprev}$, $\Vm{\tm},\Vp{\tm} \! \! \in  \! \realnums{\nx \times \kprev}$, and $\Vm{\tm}^T \! \! \! \in \! \range{\Cm{\tm}}$, $\Vp{\tm}^T \! \in \! \range{\Cp{\tm}}$, respectively. Then it holds for all $t \!\in \! \intset{0}{\tm-1}$ that $P_t \in \possemidefmats^{\nx}$ is modified as
\begin{equation}
\begin{cases}
\Pt{t} = P_t -\Vm{t} {\Cmpinv{t}}\Vm{t}^T \in \possemidefmats^{\nx} \textrm{ (downdate)} \\
\Pt{t} = P_t + \Vp{t} { \Cppinv{t}} \Vp{t}^T \in \possemidefmats^{\nx} \textrm{ (update)}
\end{cases}
\end{equation}
with $\Cm{t},\Cp{t} \in \possemidefmats^{\kprev}$, $\Vm{t},\Vp{t} \in \realnums{\nx \times \kprev}$, and $\Vm{t}^T\in \range{\Cm{t}}$, $\Vp{t}^T \in \range{\Cp{t}}$, respectively.
\end{theorem}
\begin{proof}
From the derivations of~\eqref{eq:upd:Pt_written_out_sub_iter_dwndate} and~\eqref{eq:upd:Pt_written_out_sub_iter_update} it follows that a modification of at most rank~$\kprev$ of $P_{t+1}$ at an arbitrary $t \in \intset{0}{N-1}$ in either of the forms~\eqref{eq:upd:Pt_compact_dwn_tm_before_add_rem} or~\eqref{eq:upd:Pt_compact_upd_tm_before_add_rem} results in a similar modification of at most rank~$\kprev$ of $P_t$. Since $P_{\tm}$ is modified as in~\eqref{eq:upd:Ptm_thm}, the proof follows by induction.
\end{proof}

The modified $\Kt{t+1}$ can be computed by solving
\begin{equation}
\Gt{t+1} \Kt{t+1} = - \Ht{t+1}^T. \label{eq:upd:GtKt_eq_Ht_subseq}
\end{equation}
For the common case where $\Gt{t+1} \in \posdefmats^{\nwt{t}}$, it is possible to use the Sherman-Morrison-Woodbury formula for efficient computations. For the details the reader is referred to, \eg{},~\cite{golub1996matrix}.
\begin{remark}
\label{rem:upd:Gt_fac_update}
The factorization of $\Gt{t+1}$ is computed by modifying the factorization of $G_{t+1}$. Hence, a solution to~\eqref{eq:upd:GtKt_eq_Ht_subseq} can be computed without having to re-factorize $\Gt{t+1}$, which requires less computations than re-solving~\eqref{eq:upd:GtKt_eq_Ht_subseq} from scratch.
\end{remark}


\subsection{Removing control input constraints from the working set}
\label{subsec:upd:remove_control_input_constr}
Assume that $P_{t+1}$ is modified as in~\eqref{eq:upd:Pt_compact_dwn_tm_before_add_rem} with a modification of dimension~$\kprev$. Furthermore, assume that $\kmod$ control input constraints that are affecting the control input at time~$t$ are removed from the working set, \ie, temporarily disregarding these constraints that previously were forced to hold. 
This affects the \uftoc problem~\eqref{eq:upd:uftoc_prob} in the same way as adding $\kmod$ new control inputs.
Note that this combination of modifications is more general than the one used in~\cite{nielsen13low_rank_updates}, where only modifications of the working set at a single time index was considered. Assume without loss of generality that the new control inputs are appended at the end of $\ufree{t}$. Then the matrices $\Bw{t}$, $\Qw{t}$ and $\Qxw{t}$ are modified as
\begin{equation}
\Bwt{t} \! \triangleq \! \begin{bmatrix}
\Bw{t}\! & \! \! b
\end{bmatrix} \!, \; \Qwt{t} \! \triangleq \! \begin{bmatrix}
\Qw{t} \!&\! \! \qw \\ \qw^T \! & \! \! \qwo
\end{bmatrix}\!, \; \Qxwt{t} \! \triangleq \! \begin{bmatrix}
\Qxw{t} \! & \!\!  \qxw
\end{bmatrix}\!,
\end{equation}
giving $\Bwt{t} \in \realnums{\nx \times (\nwt{t} + \kmod)}$, $\Qwt{t} \in \possemidefmats^{\nwt{t} + \kmod}$ and $\Qxwt{t} \in \realnums{\nx \times (\nwt{t} + \kmod)}$. From lines~\ref{alg:upd:ric_fac:line:G_comp}-\ref{alg:upd:ric_fac:line:H_comp} in Algorithm~\ref{alg:upd:ric_fac} it follows that $\Gt{t+1} \in \possemidefmats^{\nwt{t} + \kmod}$ and $\Ht{t+1} \in \realnums{\nx \times (\nwt{t} + \kmod)}$ are given by
\begin{subequations}
\label{eq:upd:Gt_Ht_dwndate}
\begin{align}
\Gt{t+1} = &\begin{bmatrix}
G_{t+1} & \g \\ \g^T & \go
\end{bmatrix} -  \Bwt{t}^T \Vm{t+1}
 {\Cmpinv{t+1}}  \Vm{t+1}^T \Bwt{t}, \label{eq:upd:Gtilde_dwndate}\\
 %
\Ht{t+1} =& \begin{bmatrix}
H_{t+1} & \h
\end{bmatrix} - A_t^T \Vm{t+1} {\Cmpinv{t+1}} \Vm{t+1}^T \Bwt{t},  \label{eq:upd:Htilde_dwndate}
\end{align}
\end{subequations}
where
\begin{equation}
\begin{bmatrix}
\h^T \! &\! \g^T & \! \go
\end{bmatrix}^T \! \triangleq\! \begin{bmatrix}
\qxw^T \! & \! \qw^T & \! \qwo
\end{bmatrix}^T \! + \begin{bmatrix}
A_t & \Bw{t} & b
\end{bmatrix}^T \! P_{t+1}b.
\end{equation}
In analogy with Section~\ref{subsec:upd:impact_on_subseq_iter}, $\Pt{t}$ can be computed as $\Pt{t} = \pschur{\Mt{t+1}}{\Gt{t+1}}$, where $\Mt{t+1}$ is computed as in~\eqref{eq:upd:Mt_written_out_sub_iter} but instead using $\Gt{t+1}$ and $\Ht{t+1}$ from~\eqref{eq:upd:Gt_Ht_dwndate}. By defining
\begin{equation}
\Mh{t+1} \! \triangleq \! \left[\begin{array}{c|ccc}
F_{t+1} & H_{t+1} & \h & A_t^T \Vm{t+1} \\
& & \\[\dimexpr-\normalbaselineskip+2pt] \hline
& & \\[\dimexpr-\normalbaselineskip+3pt]
H_{t+1}^T & G_{t+1} & \g & \Bw{t}^T\Vm{t+1} \\ \h^T & \g^T & \go & b^T \Vm{t+1} \\ \Vm{t+1}^T A_t & \Vm{t+1}^T \Bw{t} & \Vm{t+1}^T b & \Cm{t+1}
\end{array} \right], \label{eq:upd:Mhat_dwn}
\end{equation}
with $\Mb{t+1}$ as the second diagonal block, it is clear that $\Mt{t+1} = \pschur{\Mh{t+1}}{\Cm{t+1}}$ and $\Gt{t+1} = \pschur{\Mb{t+1}}{\Cm{t+1}}$. Hence, as in Section~\ref{subsec:upd:impact_on_subseq_iter}, Lemma~\ref{lem:pseude_schur_quotient} states that $\Pt{t} = \pschur{\Mh{t+1}}{\Mb{t+1}}$. Using the partitioning $A=F_{t+1}$ and $D=G_{t+1}$ ($B$, $C$, $E$ and $F$ consistently) of $\Mh{t+1}$, the second equality in Lemma~\ref{lem:pseude_schur_quotient} gives
\begin{equation}
\begin{split}
&\Pt{t} = \underbrace{F_{t+1} - H_{t+1}\pinv{G_{t+1}}H_{t+1}^T}_{P_{t}} - \Vm{t} {\Cmpinv{t}}\Vm{t}^T \in \possemidefmats^{\nx}, \label{eq:upd:Pt_compact_dwndate_tm}
\end{split}
\end{equation}

\noindent where 
\begin{subequations}
\begin{align}
\!\Vm{t} &\triangleq \! \begin{bmatrix}
\h \!- \!H_{t+1}\pinv{G_{t+1}}\g \!& \! \parens{A_t^T \! \! -\! H_{t+1} \pinv{G_{t+1}} \Bw{t}^T} \! \Vm{t+1}
\end{bmatrix} \! \!\label{eq:upd:def_Vt_dwn}\\
\!\Cm{t} &\!\triangleq \!\begin{bmatrix}
\go \! \! & \!\! b^T \Vm{t+1} \\
\Vm{t+1}^T b \! \! & \! \! \Cm{t+1}
\end{bmatrix}\! -\! \begin{bmatrix}
\g^T \\ \Vm{t+1}^T \Bw{t}
\end{bmatrix} \! \pinv{G_{t+1}} \!\begin{bmatrix}
\g^T \\ \Vm{t+1}^T \Bw{t}
\end{bmatrix}^T\!\!\!\!, \label{eq:upd:def_Ct_dwn} \\
\Cm{t} &\in \possemidefmats^{\kmod + \kprev}, \quad  \Vm{t}^T \in \range{\Cm{t}}.
\end{align}
\end{subequations}
Hence, removing $\kmod$ control input constraints at time $t$ from the working set when a modification in the form~\eqref{eq:upd:Pt_compact_dwn_tm_before_add_rem} of $P_{t+1}$ is already present results in a modification of $P_{t}$ in the same form as~\eqref{eq:upd:Pt_compact_dwn_tm_before_add_rem} but of increased dimension $\kprev + \kmod$. 
The modified version ${\Kt{t+1} \in \realnums{(\nwt{t} + \kmod)\times \nx}}$ can be computed by solving~\eqref{eq:upd:GtKt_eq_Ht_subseq} but using $\Gt{t+1}$ and $\Ht{t+1}$ from~\eqref{eq:upd:Gt_Ht_dwndate} instead of~\eqref{eq:upd:def_Ft_Gt_Ht_subseq_dwn}.

%

\begin{remark}
\label{rem:upd:mod_larger_than_nx}
Note that if $\kprev+\kmod $ is close to, or larger than, $\nx$ it might be better to re-compute the factorization. This trade-off can be investigated off-line by modifying the factorization for different sizes of modifications and determine which alternative is faster, but the details are left as future work.
\end{remark}

\begin{remark}
\label{rem:upd:rem_constr_no_mod_of_Ptp1}
If there is no modification of $P_{t+1}$, then $\Cm{t} \triangleq \go - \g^T \pinv{G_{t+1}} \g \in \possemidefmats^{\kmod}$ and $\Vm{t}\triangleq \h-H_{t+1}\pinv{G_{t+1}}\g \in \realnums{\nx \times \kmod}$.
\end{remark}

%
%
%
%
\begin{remark}
\label{rem:upd:sherman_morrison}
For the common case when $\Gt{t+1} \in \posdefmats^{\nwt{t}+\kmod}$, low-rank modifications can be exploited by using the Sherman-Morrison-Woodbury formula for efficient computations. The factorization of $\Gt{t+1}$ is modified as is mentioned in Remark~\ref{rem:upd:Gt_fac_update}.
\end{remark}


When removing $\kmod$ constraints from the working set also $\kmod$ components of $\ufixed{t}$ are removed. Hence, also straightforward modifications of $\Bv{t}$, $\Qxv{t}$, $\Qv{t}$, $\Qwv{t}$, $\bar{l}_{v,t}$ are made. However, these changes do not affect the matrices in the factorization and are not presented here, see~\cite{nielsen15:licthesis}.


\subsection{Adding control input constraints to the working set}
\label{subsec:upd:add_control_input_constr}

Assume that $P_{t+1}$ is modified as in the form~\eqref{eq:upd:Pt_compact_upd_tm_before_add_rem}, and that $\kmod$ control input constraints that are affecting the control input at time~$t$ are added to the working set at AS iteration $j$. Adding constraints corresponds to removing these control inputs from the problem and treating them as constants. The impact from this modification on $P_t$ is similar to when constraints are removed. Assume, without loss of generality, that the $\kmod$ control inputs are removed from the $\kmod$ last entries of $\ufree{t}$. The modified matrices $\Bwt{t}$, $\Qwt{t}$ and $\Qxwt{t}$ are then obtained from
\begin{equation}
\Bw{t} \! = \! \begin{bmatrix}
\Bwt{t}\! &\! \!b
\end{bmatrix}\!, \; \Qw{t}\! =\! \begin{bmatrix}
\Qwt{t}\! & \!\!\qw \\ \qw^T \! &\!\! \qwo
\end{bmatrix}\!, \; \Qxw{t} \! = \!\begin{bmatrix}
\Qxwt{t} \! & \!\!  \qxw
\end{bmatrix}.
\end{equation}
The implicit relations between $\Ft{t+1}$, $\Gt{t+1}$, $\Ht{t+1}$, $F_{t+1}$, $G_{t+1}$ and $H_{t+1}$ are therefore given by
\begin{subequations}
\label{eq:upd:add_constr_def_tilde_matrices}
\begin{align}
F_{t+1} &= \Ft{t+1}  - A_t^T \Vp{t+1}{\Cppinv{t+1}}\Vp{t+1}^T, \\
 \!G_{t+1} &  =\! \!  \begin{bmatrix}
\Gt{t+1} \! & \! \gt \\ \gt^T \! & \! \got 
\end{bmatrix} \!- \! \begin{bmatrix}
\Bwt{t}^T \Vp{t+1} \\ b^T \Vp{t+1}
\end{bmatrix} \! {\Cppinv{t+1}} \! \begin{bmatrix}
\Bwt{t}^T \Vp{t+1} \\ b^T \Vp{t+1}
\end{bmatrix}^T \! \! \! \!  , \label{eq:upd:def_Gtilde_upd} \\
 H_{t+1} &= \begin{bmatrix}
\Ht{t+1} & \htt 
\end{bmatrix} - A_t^T \Vp{t+1} {\Cppinv{t+1}} \begin{bmatrix}
\Bwt{t}^T \Vp{t+1} \\ b^T \Vp{t+1}
\end{bmatrix}^T. \label{eq:upd:def_Htilde_upd}
\end{align}
\end{subequations}
$\gt$, $\got$ and $\htt$ are computed from $\g$, $\go$ and $\h$ in $G_{t+1}$ and $H_{t+1}$.
Note that the modified matrices are on the right hand side.

Here, $\Mh{t+1}$ and $\Mb{t+1}$ are defined analogously as in~\eqref{eq:upd:Mhat_dwn}, but using the matrices in~\eqref{eq:upd:add_constr_def_tilde_matrices}. Hence, from Lemma~\ref{lem:pseude_schur_quotient}
\begin{align}
&P_t = \pschur{\Mh{t+1}}{\Mb{t+1}} = \underbrace{\Ft{t+1} - \Ht{t+1} \pinv{\Gt{t+1}} \Ht{t+1}^T}_{\Pt{t}} - \Vp{t}{\Cppinv{t}}\Vp{t}^T \notag \\[-5pt]
&\iff \Pt{t} = P_t + \Vp{t}{\Cppinv{t}}\Vp{t}^T \in \possemidefmats^{\nx}, \label{eq:upd:Pt_written_out_update_tm}
\end{align}
where
\begin{subequations}
\begin{align}
\!\Vp{t} &\triangleq \! \begin{bmatrix}
\htt\!-\!\Ht{t+1}\pinv{\Gt{t+1}}\gt &   \parens{A_t^T \!\!-\! \Ht{t+1} \pinv{\Gt{t+1}} \Bwt{t}^T} \! \Vp{t+1}
\end{bmatrix} \!\! \label{eq:upd:def_Vt_upd} \\
\!\Cp{t} \!&\triangleq \!\begin{bmatrix}
\got \! \! & \!\! b^T \Vp{t+1} \\
\Vp{t+1}^T b \! \! & \! \! \Cp{t+1}
\end{bmatrix}\! -\! \begin{bmatrix}
\gt^T \\ \Vp{t+1}^T \Bwt{t}
\end{bmatrix} \! \pinv{\Gt{t+1}} \!\begin{bmatrix}
\gt^T \\ \Vp{t+1}^T \Bwt{t}
\end{bmatrix}^T\!\!\!\!, \label{eq:upd:def_Ct_upd} \\
\Cp{t} &\in \possemidefmats^{\kmod + \kprev}, \quad  \Vp{t}^T \in \range{\Cp{t}}.
\end{align}
\end{subequations}
Hence, adding $\kmod$ control input constraints at time $t$ to the working set when a modification in the form~\eqref{eq:upd:Pt_compact_upd_tm_before_add_rem} is already present results in a modification of $P_{t}$ in the same form as~\eqref{eq:upd:Pt_compact_upd_tm_before_add_rem} but of increased dimension $\kprev + \kmod$. The modified $\Kt{t+1}$ can be computed by solving~\eqref{eq:upd:GtKt_eq_Ht_subseq}, but using the modified matrices in~\eqref{eq:upd:add_constr_def_tilde_matrices}. Note that Remark~\ref{rem:upd:mod_larger_than_nx} and~\ref{rem:upd:sherman_morrison} apply here as well.
\begin{remark}
\label{rem:upd:add_constr_no_mod_of_Ptp1}
If there is no modification of $P_{t+1}$, then $\Cp{t} \triangleq \go - \g^T \pinv{\Gt{t+1}} \g \in \possemidefmats^{\kmod}$ and $\Vp{t}\triangleq \h-\Ht{t+1}\pinv{\Gt{t+1}}\g \in \realnums{\nx \times \kmod}$.
\end{remark}

\subsection{Algorithms for modifying the Riccati factorization}
\label{subsec:upd:algorithms}
%
Let $\tm$ be the largest time index where $\wset{j}$ is modified. 
The theory presented in this section is summarized in Algorithm~\ref{alg:upd:ric_fac_modification}, which starts by modifying the matrices in the factorization according to Section~\ref{subsec:upd:remove_control_input_constr} or Section~\ref{subsec:upd:add_control_input_constr} depending on whether constraints are removed or added to the working set, respectively. Since $P_{\tm+1}$ is not modified Remark~\ref{rem:upd:rem_constr_no_mod_of_Ptp1} or Remark~\ref{rem:upd:add_constr_no_mod_of_Ptp1}, respectively, applies.  For $t < \tm$ the matrices in the factorization are modified as in Section~\ref{subsec:upd:impact_on_subseq_iter},~\ref{subsec:upd:remove_control_input_constr} or~\ref{subsec:upd:add_control_input_constr} depending on the type of modification at time $t$. Note that only adding \emph{or} removing constraints is possible at the same \AS iteration. As is mentioned in Remark~\ref{rem:upd:Gt_fac_update} standard methods for modifying the factorization of $G_{t+1}$ should be used to avoid re-computing the factorization. See for example~\cite{golub1996matrix,stewart1998:matrix,nocedal06:num_opt} for details on techniques for modifying factorizations.

For an example with non-singular $\Qu{t}$ where $\kmod$ constraints are removed at time $\tm$ and Cholesky factorizations of $G_{t+1}$ are used, the computational complexity when modifying the Riccati factorization instead of re-computing it is reduced from approximately $\Ordo{N(\nw^3+\nx^3+\nw^2\nx+\nx^2\nw)}$ to approximately $\Ordo{\tm(\nw^2\nx + \nw^2 + k\nw \nx+ k \nx^2)}$. 
If the solution to~\eqref{eq:upd:GtKt_eq_Ht_subseq} is computed using the Sherman-Morrison-Woodbury formula the complexity is further reduced to approximately $\Ordo{\tm(\nw^2 + k \nw^2 + k \nw \nx + k \nx^2)}$. 
Note that the complexity is now linear in $\tm$ and quadratic in $\nx$ and $\nw$, which shows the gains of modifying the Riccati factorization instead of re-computing it. However, the exact expression for the complexity depends on for example the choice of factorization and modification techniques in Algorithm~\ref{alg:upd:ric_fac} and~\ref{alg:upd:ric_fac_modification}.

\begin{algorithm}[h!]
  \caption{Modification of the Riccati factorization}
  \label{alg:upd:ric_fac_modification}
  \begin{algorithmic}[1]
  		\STATE Set $\tm$ as the largest $t$ for which $\wset{j}$ is modified
  		\IF{Constraints are removed at time $\tm$}
  			\STATE Modify factorization as in Section~\ref{subsec:upd:remove_control_input_constr} using Remark~\ref{rem:upd:rem_constr_no_mod_of_Ptp1} 
  		\ELSIF{Constraints are added at time $\tm$}
  			\STATE Modify factorization as in Section~\ref{subsec:upd:add_control_input_constr} using Remark~\ref{rem:upd:add_constr_no_mod_of_Ptp1}
  		\ENDIF
  	\FOR{$t=\tm-1,\ldots,0$}
		\IF{No modification of $\wset{j}$ at time $t$}
			\STATE Modify factorization as in Section~\ref{subsec:upd:impact_on_subseq_iter}
		\ELSIF{Constraints are removed at time $t$}
				\STATE Modify factorization as in Section~\ref{subsec:upd:remove_control_input_constr}
       		\ELSIF{Constraints are added at time $t$}
       			\STATE Modify factorization as in Section~\ref{subsec:upd:add_control_input_constr}
       \ENDIF
	\ENDFOR
	\end{algorithmic}
\end{algorithm}

To compute the solution to the modified \uftoc problem, the recursions in algorithms~\ref{alg:upd:ric_rec_bw_rec}-\ref{alg:upd:dual_var_fw_rec} need to be re-computed. Since the factorization is only modified for $t \in \intset{0}{\tm}$ the backward recursion in Algorithm~\ref{alg:upd:ric_rec_bw_rec} only needs to be re-computed for $t \leq \tm$ using the modified matrices $\Bwt{t}$, $\Gt{t}$, $\Ht{t}$ and $\Pt{t}$.




\section{Extension to General Constraints}
\label{sec:upd:general_constr}

 The \cftoc{} problem arising in many \MPC problems in industry often includes constraints on the states and the possibility to control only certain states or a linear combination of states~\cite{maciejowski2002predictive}, and is of a more general type of problem than~\eqref{eq:upd:primal_problem}. Here it will be described how the theory presented in Section~\ref{sec:upd:lr_modifications} can be used to compute the search directions even when solving more general \cftoc problems than~\eqref{eq:upd:primal_problem} using an \AS type solver. Note that the purpose with this section is not to present a complete \AS solver, but to explain how the theory can be used when solving more general problems than~\eqref{eq:upd:primal_problem}. In this section, the superscripts ``$p$'' and ``$d$'' denote variables related to the primal problem and the dual problem, respectively.


\subsection{Primal and dual \cftoc problems}
\label{subsec:upd:primal_and_dual_problem}

Consider a \cftoc problem with states $x_t^p \in \realnums{\nx}$, controlled variables $\ep{t} \in \realnums{\nee}$ and control inputs $u_t^p\in \realnums{\nuut{t}}$, and with inequality constraints on both states and control inputs. This general type of \cftoc problems covers many linear \MPC applications, and is given by the optimization problem
 \begin{equation}
 \minimizes{&\sum_{t=0}^{N-1} \Big( \frac{1}{2} \begin{bmatrix}
 \ep{t} \\ \up{t}
\end{bmatrix}^T \Qpartp{t}  \begin{bmatrix}
\ep{t} \\ \up{t}
\end{bmatrix} + \begin{bmatrix}
\linep{t} \\ \linup{t}
\end{bmatrix}^T  \begin{bmatrix}
\ep{t} \\ \up{t}
\end{bmatrix} + \cpartp{t} \Big)+ \\ & \frac{1}{2}\epT{N} \Qep{N} \ep{N} + \linepT{N} \ep{N} + \cpartp{N}  }{\timestack{\xp{}}, \timestack{\ep{}},\timestack{\up{}}}{&\xp{0} = \xinit \\ &\xp{t+1} = \Ap{t}\xp{t} +\Bp{t}\up{t} + \ap{t}, \; t \in \intset{0}{N-1} \\ &\ep{t} = M_t^p \xp{t}, \; t \in \intset{0}{N} \\ &\ineqx{t}^p \xp{t} + \inequ{t}^p\up{t} + \ineqaff{t}^p \preceq 0, \; t \in \intset{0}{N-1} \\ &\ineqx{N}^p\xp{N} + \ineqaff{N}^p \preceq 0,  } \label{eq:upd:cftoc_state_constr_primal}
 \end{equation}
 where $\ineqx{t}^p \in \realnums{\nc{t} \times \nx}$, $\inequ{t}^p \in \realnums{\nc{t} \times \nuut{t}}$ and $\ineqaff{t}^p \in \realnums{\nc{t}}$ defines the $\nc{t}$  inequality constraints at time $t$, and 
$\Qpart{t} \in \posdefmats^{\nee + \nuut{t}}$. 
Furthermore, let $\alpha_{t} \in \realnums{\nx}$, $\beta_t \in \realnums{\nee}$ and $\gamma_t \in \realnums{\nc{t}}$ (for all $t \in \intset{0}{N}$) be the dual variables for the dynamics constraints, the constraints $-\ep{t} + M_t^p \xp{t} = 0$, and the inequality constraints in~\eqref{eq:upd:cftoc_state_constr_primal}, respectively.

It is shown in~\cite{axehill06:_mixed_integ_dual_quadr_progr_tail_mpc,nielsen15:licthesis} that the dual problem to~\eqref{eq:upd:cftoc_state_constr_primal} can also be interpreted as a \cftoc problem with the structure
\begin{equation}
\minimizes{&\sum_{\tau = 0}^{\Nd-1}\parens{ \frac{1}{2} \begin{bmatrix}
\xd{\tau} \\ \ud{\tau}
\end{bmatrix}^T \Qpartd{\tau} \begin{bmatrix}
\xd{\tau} \\ \ud{\tau}
\end{bmatrix} +  \begin{bmatrix}
\linxd{\tau} \\ \linud{\tau}
\end{bmatrix}^T \begin{bmatrix}
\xd{\tau} \\ \ud{\tau}
\end{bmatrix} + \cdual{\tau} } + \\&   \linxdT{\Nd} \xd{\Nd} }{\timestack{\xd{},\ud{}}}{&\xd{0} = 0 \\ &\xd{\tau + 1} = \Ad{\tau} \xd{\tau} + \Bd{\tau} \ud{\tau}, \; \tau \in \intset{0}{\Nd-1} \\ &\begin{bmatrix}
0 & -I_{\nc{\Nd-1-\tau}}
\end{bmatrix} \ud{\tau} \preceq 0, \; \tau \in \intset{0}{\Nd-1} ,}\label{eq:upd:cftoc_state_constr_dual}
\end{equation}
where $\tau \triangleq N-t$, $\Nd \triangleq N+1$, the state variables $\xd{\tau} \in \realnums{\nx}$ and control inputs $\ud{\tau} \in \realnums{\nee+\nc{N-\tau}}$ are introduced as
\begin{equation}
\xd{\tau} \!\triangleq \! \alpha_{N+1-\tau}, \, \tau \! \in\! \intset{1}{N+1}, \; \;
\ud{\tau} \triangleq \begin{bmatrix}
\beta_{N-\tau} \\ \gamma_{N-\tau}
\end{bmatrix}, \, \tau \!\in\! \intset{0}{N}, \label{eq:upd:def_xd_ud}
\end{equation}
and the quadratic terms in the objective function satisfy
\begin{equation}
\Qpartd{\tau} \! \in  \possemidefmats^{\nx + \nee + \nc{\Nd-1-\tau}}, \; \tau \!\in \intset{0}{\Nd-1}, \; \Qxd{\Nd} \in \possemidefmats^{\nx}.
\end{equation}
Note that there are no state constraints in the dual problem~\eqref{eq:upd:cftoc_state_constr_dual} despite that~\eqref{eq:upd:cftoc_state_constr_primal} has it, and that $\Qpartd{\tau}$ is positive semidefinite. 
%


Once the dual problem has been solved, the primal variables can be computed from the dual solution using the equations
\begin{subequations}
\label{eq:upd:relation_primal_dual_sol}
\begin{align}
\xp{t} &= - \lambda_{\Nd-t}, \; t \in \intset{0}{N}, \label{eq:upd:xp_lamda_relation}  \\
\up{t} &= - \QeubpT{t}\linep{t} - \Qubp{t} \linup{t} -  \Qubp{t} \BpT{t} \xd{N-t} + \notag \\ &  \begin{bmatrix}
\QeubpT{t} & - \Qubp{t} \parens{\inequ{t}^p}^T
\end{bmatrix}\ud{N-t}, \; t \in \intset{0}{N-1},\label{eq:upd:up_from_dual}
\end{align}
\end{subequations}
where $\lambda_\tau$ are the dual variables corresponding to the equality constraints in the dual problem~\eqref{eq:upd:cftoc_state_constr_dual}, and
\begin{equation}
\begin{bmatrix}
\Qebp{t} & \Qeubp{t} \\ \QeubpT{t} & \Qubp{t}
\end{bmatrix} \! \triangleq \! \inv{\begin{bmatrix}
\Qep{t} & \Qeup{t} \\ \QeupT{t} & \Qup{t}
\end{bmatrix}} \!\! \!\!  =  \inv{(\Qpartp{t})}. \label{eq:upd:def_Qb}
\end{equation}

\subsection{Computing the search direction in the dual}
\label{subsec:method_gen_constr}

\vspace{-2pt}

One possibility to handle state constraints is to solve the primal problem~\eqref{eq:upd:cftoc_state_constr_primal} using for example a dual \AS type solver as proposed in~\cite{axehill06:_mixed_integ_dual_quadr_progr_tail_mpc}, or a dual gradient projection method as in~\cite{axehill08:thesis,axehill08:_dual_gradien_projec_quadr_progr}. In these types of methods, the primal problem~\eqref{eq:upd:cftoc_state_constr_primal} is solved by computing the solution to the corresponding dual problem~\eqref{eq:upd:cftoc_state_constr_dual} using primal methods. The dual \cftoc problem~\eqref{eq:upd:cftoc_state_constr_dual} is in the same form as the \cftoc problem~\eqref{eq:upd:primal_problem} which has only simple bounds on the control input. Hence, it is solved by computing a sequence of search directions corresponding to the solutions of \uftoc problems in the form~\eqref{eq:upd:uftoc_prob}. If an \AS type solver employing Riccati recursions is used, the theory presented in this paper directly applies. The primal solution to~\eqref{eq:upd:cftoc_state_constr_primal} is obtained from~\eqref{eq:upd:relation_primal_dual_sol}.


However, when a dual solver is used to solve~\eqref{eq:upd:cftoc_state_constr_primal}, primal feasibility is obtained only at the optimum~\cite{goldfarb83:numerical_stable_dual_method_for_solving_QP,bartlett06:QPSchur}. In a real-time \MPC control loop this might be problematic since the computed control input is not necessarily primal feasible due to early termination to satisfy real-time constraints.
An approach to address this problem and still be able to perform low-rank modifications of the Riccati factorization with state-constraints present is presented here. The idea is to use a primal solver which maintains primal feasibility, but that computes the search direction by solving a dual \uftoc problem.  This can be done by exploiting the relation between the working sets and variables in the primal problem~\eqref{eq:upd:cftoc_state_constr_primal} and in the dual problem~\eqref{eq:upd:cftoc_state_constr_dual}, respectively. 

To do this, let $\ineqx{i,t}^p$ denote the $i$:th row of $\ineqx{t}^p$, and let the notation $(i,t) \in \wset{j}$ indicate that the inequality constraint
 \begin{equation}
 \ineqx{i,t}^p \xp{t} + \inequ{i,t}^p \up{t} + \ineqaff{i,t}^p \leq 0,
 \end{equation}
 \vspace{-15pt}
 
 \noindent is in the working set and is thus forced to hold with equality. The primal search direction at \AS iteration $j$ is computed by solving the equality constrained problem (in compact notation)
 \begin{equation}
 \minimizes{J^p(\timestack{\xp{}},\timestack{\ep{}},\timestack{\up{}})}{\timestack{\xp{}},\timestack{\ep{}},\timestack{\up{}}}{&\timestack{A^p}\timestack{\xp{}}+\timestack{B^p}\timestack{\up{}}=\timestack{a^p} \\ &\timestack{\ep{}} = \timestack{M^p}\timestack{\xp{}} \\ &\ineqx{i,t}^p \xp{t} \! + \! \inequ{i,t}^p \up{t} + \ineqaff{i,t}^p = 0, \; (i,t) \!\in \! \wset{j},} \label{eq:upd:primal_eg_constr_QP}
 \end{equation}
 
 \noindent where $J^p$ is the objective function in~\eqref{eq:upd:cftoc_state_constr_primal} and the two first equality constraints are the equality constraints in~\eqref{eq:upd:cftoc_state_constr_primal} presented in compact notation. Note that $\gamma_{i,t}$ for $(i,t) \in \wset{j}$ are unconstrained, and $\gamma_{i,t} = 0$ for all $(i,t) \in \wsetc{j}$,~\cite{nocedal06:num_opt}.
Hence, from the definition of $\ud{\tau}$ in~\eqref{eq:upd:def_xd_ud} it follows that $\ud{\nee+i,N-t} = \gamma_{i,t}$ for all $(i,t) \in \wset{j}$ are unconstrained optimization variables in the dual problem, and $\ud{\nee + i,N-t}\! =\!\gamma_{i,t}\! =\!  0$ for all $(i,t) \!\in\! \wsetc{j}$. Hence, instead of solving~\eqref{eq:upd:primal_eg_constr_QP} directly, the solution can be {computed by solving the corresponding dual problem}
\begin{equation}
\minimizes{J^d(\timestack{\xd{}},\timestack{\ud{}})}{\timestack{\xd{}},\timestack{\ud{}}}{ &\timestack{A^d}\timestack{\xd{}} + \timestack{B^d}\timestack{\ud{}} = \timestack{a^d} \\ &\ud{\nee + i,N-t} = 0, \; (i,t) \in \wsetc{j}, } \label{eq:upd:dual_eq_constr_QP}
\end{equation}
 
 \noindent where $J^d$ is the objective function in~\eqref{eq:upd:cftoc_state_constr_dual} and the equality constraints in~\eqref{eq:upd:cftoc_state_constr_dual} are compactly written as the first constraint in~\eqref{eq:upd:dual_eq_constr_QP}. The primal solution  is obtained from~\eqref{eq:upd:relation_primal_dual_sol}. By eliminating the constrained dual control inputs,~\eqref{eq:upd:dual_eq_constr_QP} is in the same \uftoc form as~\eqref{eq:upd:uftoc_prob}.
Furthermore, removing a constraint from the working set in the primal problem corresponds to adding a constraint in the dual problem, \ie, constrain one dual control input, and vice versa. Hence, the structure of the modifications of the dual \uftoc problem between \AS iterations are the same as for the \uftoc problem~\eqref{eq:upd:uftoc_prob}, and the theory presented in Section~\ref{sec:upd:lr_modifications} can be used to modify the Riccati factorization when solving a sequence of problems in the form~\eqref{eq:upd:dual_eq_constr_QP}.

\vspace{-10pt}

\section{Numerical Results}
\label{sec:upd:numerical_results}
\vspace{-7pt}

In this section, the proposed algorithm for solving the \KKT system of~\eqref{eq:upd:uftoc_prob} by modifying the Riccati factorization is compared to the standard Riccati recursion. A proof-of-concept implementation is made in \matlab, where most of the main operations such as Cholesky factorizations have been implemented in m-code to get a fair comparison of the computational times. In this implementation the gaxpy Cholesky in~\cite{golub1996matrix} and the Cholesky modifications from~\cite{stewart1998:matrix} are used. 
The m-code is converted into \textsc{C} code by using \matlab's code generation framework, and the generated \textsc{C} code is used to produce the numerical results. As always, to get a completely fair comparison of the algorithms, fully optimized implementations in a compiled language should be used. However, this is outside the scope of this paper.

All computations were performed on an Intel Xeon W3565 @3.2 GHz processor running Linux (version 2.6.32-504.12.2.el6.x86\_64) and \matlab (version 9.1.0.441655 (R2016b)). The default settings have been used for the code generation in \matlab, with the exception that the compilation flag '-O3' has been used to optimize the code for speed.





The algorithms are compared by solving random \uftoc{} problems in the form~\eqref{eq:upd:uftoc_prob}, where $\nx$ and $\nw$ are logarithmically spaced in $\intset{10}{200}$. The computation times are averaged over 20 different problems of the same dimensions. In Fig.~\ref{fig:upd:nx_vs_nu_3d} the computation times are normalized w.r.t. to the maximum computation time $0.65$ seconds for the standard Riccati recursion. Here, a problem with $N\!=\!10$ has been solved after removing a constraint at either $\tm\!=\!0$ (modifying one step of the factorization) or {$\tm \! = \! N\!-\!1$ (modifying the full factorization)}, which are the best and worst case for the modifying algorithm, respectively. Furthermore, in Fig.~\ref{fig:upd:N_vs_tm} the performance gains for different $N$ and $\tm \in \{1,\frac{N}{4}, \frac{N}{2}, \frac{3N}{4},N \}$ are investigated by plotting the ratio between the computation times when modifying the Riccati factorization and re-computing it for problems of dimension $\nx,\nw\!=\!10$, $\nx,\nw\!=\!103$ and $\nx,\nw\!=\!200$, respectively. From the figures it is clear that modifying the Riccati factorization instead of re-computing it can significantly reduce the computation time for solving the \uftoc problem~\eqref{eq:upd:uftoc_prob}, especially for large problem sizes and/or when only a small part of the factorization is modified. The $\Ordo{\tm}$ complexity (independent on $N$) result in Section~\ref{subsec:upd:algorithms} is numerically verified in Fig.~\ref{fig:upd:N_vs_tm}, where it is shown that the performance is similar for $N \in \{10,20,40,60,80,100\}$. The accuracies of the numerical solutions have been measured as the Euclidean norm of the \KKT residual for the \uftoc problem~\eqref{eq:upd:uftoc_prob}. For a problem with $N=100$ and $\nx,\nw = 200$ the maximum residual norm is in the order $10^{-10}$ for both the standard Riccati recursion and the proposed algorithm.

\begin{figure}
\centering
\begin{minipage}{0.9\columnwidth}
\input{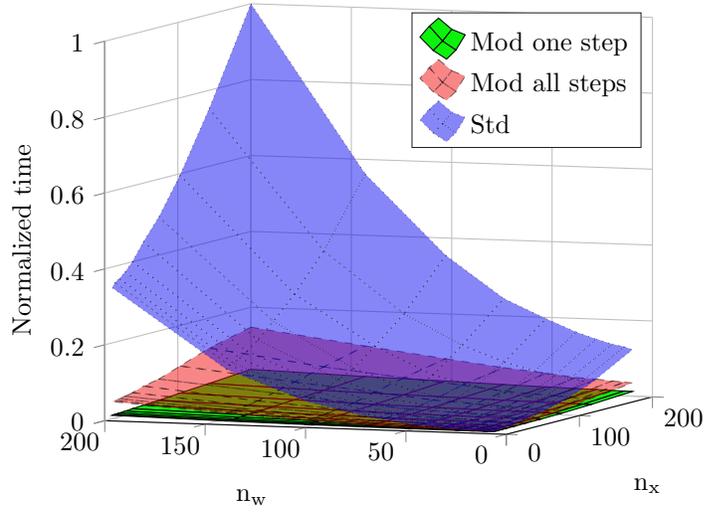}
\end{minipage}
\caption{Normalized (w.r.t. the maximum computation time) computation time for computing the Riccati recursion using the standard (Algorithm~\ref{alg:upd:ric_fac}) and modifying (Algorithm~\ref{alg:upd:ric_fac_modification}) algorithms, respectively, when one constraint is removed. Here the case where only one step (best case) and all steps (worst case) are modified are shown.}
\label{fig:upd:nx_vs_nu_3d}
\end{figure}

\vspace{-10pt}

\begin{figure}
\centering
\begin{minipage}{0.9\columnwidth}
\input{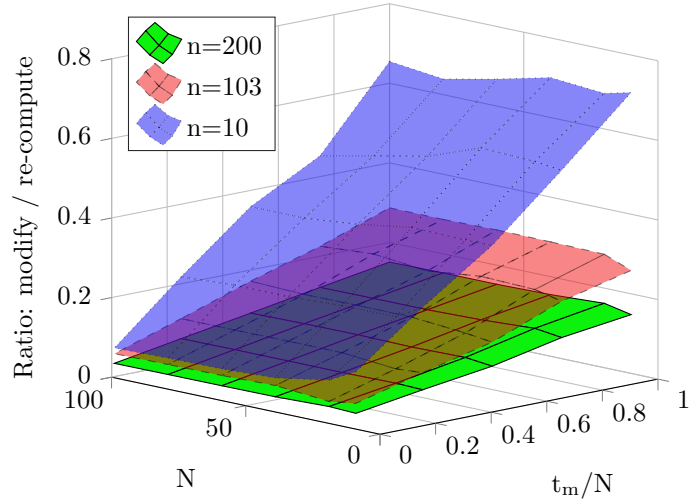}
\end{minipage}
\caption{Computation time ratio between computing the Riccati recursion using Algorithm~\ref{alg:upd:ric_fac_modification}  and using Algorithm~\ref{alg:upd:ric_fac} for problems with $\nx,\nw=200$, $\nx,\nw=103$ and $\nx,\nw=10$. }
\label{fig:upd:N_vs_tm}
\end{figure}

\section{Conclusions}
\label{sec:upd:conclusions}
This work presents theory and algorithms for modifying the Riccati factorization instead of re-computing it after low-rank modifications of the \KKT system have been made. This is possible by exploiting the special structure from the \MPC problem, and it can be used to significantly improve the performance of \AS type solvers by modifying the Riccati factorization between \AS iterations instead of re-computing it. The algorithm has been evaluated using a \textsc{C} implementation generated from \matlab's code generation framework, and it is shown that significant gains in terms of performance can be obtained using the proposed algorithm. The result shows that Riccati recursions can be employed in \AS methods without sacrificing the important possibility to exploit low-rank modifications of the \KKT systems when computing the search directions required to solve a \cftoc problem.


\newpage

\bibliography{IEEEfull,ianFull,axe_full}

\begin{thebibliography}{10}
\providecommand{\url}[1]{#1}
\csname url@samestyle\endcsname
\providecommand{\newblock}{\relax}
\providecommand{\bibinfo}[2]{#2}
\providecommand{\BIBentrySTDinterwordspacing}{\spaceskip=0pt\relax}
\providecommand{\BIBentryALTinterwordstretchfactor}{4}
\providecommand{\BIBentryALTinterwordspacing}{\spaceskip=\fontdimen2\font plus
\BIBentryALTinterwordstretchfactor\fontdimen3\font minus
  \fontdimen4\font\relax}
\providecommand{\BIBforeignlanguage}[2]{{%
\expandafter\ifx\csname l@#1\endcsname\relax
\typeout{** WARNING: IEEEtran.bst: No hyphenation pattern has been}%
\typeout{** loaded for the language `#1'. Using the pattern for}%
\typeout{** the default language instead.}%
\else
\language=\csname l@#1\endcsname
\fi
#2}}
\providecommand{\BIBdecl}{\relax}
\BIBdecl

\bibitem{maciejowski2002predictive}
J.~Maciejowski, \emph{Predictive control with constraints}.\hskip 1em plus
  0.5em minus 0.4em\relax Prentice Hall, 2002.

\bibitem{jonson83:thesis}
H.~Jonson, ``A {N}ewton method for solving non-linear optimal control problems
  with general constraints,'' Ph.D. dissertation, Link\"{o}pings Tekniska
  H\"{o}gskola, 1983.

\bibitem{rao98:_applic_inter_point_method_model_predic_contr}
C.~Rao, S.~Wright, and J.~Rawlings, ``Application of interior-point methods to
  model predictive control,'' \emph{Journal of Optimization Theory and
  Applications}, vol.~99, no.~3, pp. 723--757, Dec. 1998.

\bibitem{hansson00:_primal_dual_inter_point_method}
A.~Hansson, ``A primal-dual interior-point method for robust optimal control of
  linear discrete-time systems,'' \emph{{IEEE} Transactions on Automatic
  Control}, vol.~45, no.~9, pp. 1639--1655, Sep. 2000.

\bibitem{vandenberghe02:_robus_full}
L.~Vandenberghe, S.~Boyd, and M.~Nouralishahi, ``Robust linear programming and
  optimal control,'' Department of Electrical Engineering, University of
  California Los Angeles, Tech. Rep., 2002.

\bibitem{axehill06:_mixed_integ_dual_quadr_progr_tail_mpc}
D.~Axehill and A.~Hansson, ``A mixed integer dual quadratic programming
  algorithm tailored for {MPC},'' in \emph{Proceedings of the 45th {IEEE}
  Conference on Decision and Control}, San Diego, USA, Dec. 2006, pp.
  5693--5698.

\bibitem{axevanhan07:_relax_applic_mipc_compa_and_eff_compu}
D.~Axehill, A.~Hansson, and L.~Vandenberghe, ``Relaxations applicable to mixed
  integer predictive control -- comparisons and efficient computations,'' in
  \emph{Proceedings of the 46th {IEEE} Conference on Decision and Control}, New
  Orleans, USA, 2007, pp. 4103--4109.

\bibitem{axehill08:thesis}
D.~Axehill, ``Integer quadratic programming for control and communication,''
  Ph.D. dissertation, Link\"{o}ping University, 2008.

\bibitem{axehill08:_dual_gradien_projec_quadr_progr}
D.~Axehill and A.~Hansson, ``A dual gradient projection quadratic programming
  algorithm tailored for model predictive control,'' in \emph{Proceedings of
  the 47th {IEEE} Conference on Decision and Control}, Cancun, Mexico, 2008,
  pp. 3057--3064.

\bibitem{diehl09:_nonlin_model_predic_contr}
M.~Diehl, H.~Ferreau, and N.~Haverbeke, \emph{Nonlinear model predictive
  control}.\hskip 1em plus 0.5em minus 0.4em\relax Springer Berlin /
  Heidelberg, 2009, ch. Efficient Numerical Methods for Nonlinear MPC and
  Moving Horizon Estimation, pp. 391--417.

\bibitem{nielsen13low_rank_updates}
I.~Nielsen, D.~Ankelhed, and D.~Axehill, ``Low-rank modifications of {R}iccati
  factorizations with applications to model predictive control,'' in
  \emph{Proceedings of the 52nd {IEEE} Conference on Decision and Control},
  Firenze, Italy, Dec. 2013, pp. 3684--3690.

\bibitem{nielsen15:parallel_factorization_cdc}
I.~Nielsen and D.~Axehill, ``A parallel structure exploiting factorization
  algorithm with applications to model predictive control,'' in
  \emph{Proceedings of the 54th {IEEE} Conference on Decision and Control},
  Osaka, Japan, Dec. 2015, pp. 3932--3938.

\bibitem{frison16:thesis}
G.~Frison and J.~J{\o}rgensen, ``Algorithms and methods for high-performance
  model predictive control,'' Ph.D. dissertation, 2016.

\bibitem{kirches2011factorization}
C.~Kirches, H.~Bock, J.~Schl{\"o}der, and S.~Sager, ``A factorization with
  update procedures for a {KKT} matrix arising in direct optimal control,''
  \emph{Mathematical Programming Computation}, vol.~3, no.~4, pp. 319--348,
  2011.

\bibitem{nocedal06:num_opt}
J.~Nocedal and S.~Wright, \emph{Numerical optimization}.\hskip 1em plus 0.5em
  minus 0.4em\relax Springer-Verlag, 2006.

\bibitem{nielsen15:licthesis}
I.~Nielsen, \emph{On structure exploiting numerical algorithms for model
  predictive control}, ser. (Licentiate's thesis) Link\"{o}ping Studies in
  Science and Technology. Thesis, 2015, no. 1727.

\bibitem{boyd04:_convex_optim}
S.~Boyd and L.~Vandenberghe, \emph{Convex optimization}.\hskip 1em plus 0.5em
  minus 0.4em\relax Cambridge University Press, 2004.

\bibitem{albert69}
A.~Albert, ``Conditions for positive and nonnegative definiteness in terms of
  pseudoinverses,'' \emph{SIAM Journal on Applied Mathematics}, vol.~17, no.~2,
  pp. 434--440, 1969.

\bibitem{carlsonPseudo}
D.~Carlson, E.~Haynsworth, and T.~Markham, ``\BIBforeignlanguage{English}{A
  generalization of the {S}chur complement by means of the {M}oore-{P}enrose
  inverse},'' \emph{\BIBforeignlanguage{English}{SIAM Journal on Applied
  Mathematics}}, vol.~26, no.~1, pp. pp. 169--175, 1974.

\bibitem{zhang2005schur}
F.~Zhang, \emph{The Schur complement and its applications}.\hskip 1em plus
  0.5em minus 0.4em\relax Springer, 2005, vol.~4.

\bibitem{golub1996matrix}
G.~Golub and C.~Van~Loan, \emph{Matrix computations}, ser. Johns Hopkins
  Studies in the Mathematical Sciences.\hskip 1em plus 0.5em minus 0.4em\relax
  Johns Hopkins University Press, 1996.

\bibitem{stewart1998:matrix}
G.~Stewart, \emph{Matrix algorithms: Volume 1, basic decompositions}, ser.
  Matrix Algorithms.\hskip 1em plus 0.5em minus 0.4em\relax Society for
  Industrial and Applied Mathematics, 1998.

\bibitem{goldfarb83:numerical_stable_dual_method_for_solving_QP}
D.~Goldfarb and A.~Idnani, ``\BIBforeignlanguage{English}{A numerically stable
  dual method for solving strictly convex quadratic programs},''
  \emph{\BIBforeignlanguage{English}{Mathematical Programming}}, vol.~27,
  no.~1, pp. 1--33, 1983.

\bibitem{bartlett06:QPSchur}
R.~Bartlett and L.~Biegler, ``\BIBforeignlanguage{English}{Qpschur: A dual,
  active-set, schur-complement method for large-scale and structured convex
  quadratic programming},'' \emph{\BIBforeignlanguage{English}{Optimization and
  Engineering}}, vol.~7, no.~1, pp. 5--32, 2006.

\end{thebibliography}

\end{document}